\documentclass[12pt]{article}
\usepackage{amsmath,amssymb,amsthm, amsfonts}
\usepackage{etaremune, enumerate, float, verbatim}
\usepackage{algorithm}
\usepackage{algorithmic}
\usepackage{hyperref}
\usepackage{graphicx}
\usepackage{url}
\usepackage[mathlines]{lineno}
\usepackage{dsfont} 
\usepackage{tikz, graphicx, subcaption, caption}
\usepackage[algo2e,ruled,vlined]{algorithm2e}
\usepackage{mathrsfs,amsmath}
\usepackage{color}
\definecolor{red}{rgb}{1,0,0}
\def\red{\color{red}}
\definecolor{blue}{rgb}{0,0,.7}

\definecolor{green}{rgb}{0,.6,0}

\definecolor{purp}{rgb}{.5,0,.5}

\numberwithin{figure}{section}   

\newtheorem{thm}{Theorem}[section]
\newtheorem{cor}[thm]{Corollary}
\newtheorem{lem}[thm]{Lemma}
\newtheorem{prop}[thm]{Proposition}

\theoremstyle{definition}
\newtheorem{rem}[thm]{Remark}

\theoremstyle{definition}

\theoremstyle{definition}
\newtheorem{ex}[thm]{Example}

\newcommand{\Z}{\operatorname{Z}}
\newcommand{\ZFS}{\operatorname{ZFS}}

\newcommand{\zfg}{\mathscr{Z}}
\newcommand{\PCN}{\operatorname{P}}
\newcommand{\PC}{\mathcal{P}}
\newcommand{\C}{\mathcal{C}}
\newcommand{\clf}{\mathscr{F}}

\newcommand{\dist}{\operatorname{dist}}

\newcommand{\x}{\times}

\newcommand{\wt}{\widetilde}
\newcommand{\bit}{\begin{itemize}}
\newcommand{\eit}{\end{itemize}}
\newcommand{\ben}{\begin{enumerate}}
\newcommand{\een}{\end{enumerate}}
\newcommand{\beq}{\begin{equation}}
\newcommand{\eeq}{\end{equation}}
\newcommand{\bea}{\begin{eqnarray*}} 
\newcommand{\eea}{\end{eqnarray*}}
\newcommand{\bpf}{\begin{proof}}
\newcommand{\epf}{\end{proof}\ms}
\newcommand{\bmt}{\begin{bmatrix}}
\newcommand{\emt}{\end{bmatrix}}
\newcommand{\ms}{\medskip}

\newcommand{\cp}{\, \Box\,}
\newcommand{\lc}{\left\lceil}
\newcommand{\rc}{\right\rceil}

\newcommand{\noi}{\noindent}

\title{Reconfiguration graphs of zero forcing sets}
\author{Jesse Geneson\thanks{Department of Mathematics and Statistics, San Jos\'e State University, One Washington Square, San Jos\'e, CA 95192, USA
(geneson@gmail.com).} \and Ruth Haas\thanks{Department of Mathematics, University of Hawai'i at M\=anoa, 2565 McCarthy Mall, Honolulu, HI 96822, USA (rhaas@math.hawaii.edu)}\and Leslie Hogben\thanks{Department of Mathematics, Iowa State University,
Ames, IA 50011, USA and American Institute of Mathematics, 600 E. Brokaw Road, San Jos\'e, CA 95112, USA
(hogben@aimath.org).}}

\begin{document}
\maketitle\vspace{-20pt}

\begin{abstract}
This paper begins the study of reconfiguration of zero forcing sets, and more specifically,  the zero forcing graph. Given a base graph $G$,   its zero forcing graph,  $\zfg(G)$,  is  the graph whose vertices are the minimum zero forcing sets of $G$ with an edge between vertices $B$ and $B'$ of $\zfg(G)$ if and only if $B$ can be obtained from $B'$ by changing a single vertex of $G$.  It is shown that the zero forcing graph of a forest is connected, but that many zero forcing graphs are disconnected. We characterize the  base graphs whose zero forcing graphs are either a path or the complete graph, and show that the star cannot be a zero forcing graph.   We show that  computing $\zfg(G)$ takes $2^{\Theta(n)}$ operations in the worst case for a graph $G$ of order $n$. \end{abstract}

\noi {\bf Keywords} reconfiguration, zero forcing, zero forcing graph

\noi{\bf AMS subject classification} 68R10, 05C50, 05C57 

\section{Introduction}

Reconfiguration is concerned with relationships among solutions to a problem instance.   Reconfiguration of one  feasible solution into another is accomplished through a sequence of steps, where each step follows a reconfiguration rule, and  such that each intermediate solution is also feasible. The reconfiguration graph is the set of all feasible solutions to the problem with an edge between two solutions if one solution 
can be obtained from the other  by one application of the reconfiguration rule. Research in reconfiguration  addresses both structural questions and algorithmic ones.

One of the most well studied reconfiguration scenarios is  vertex coloring where all 
proper colorings  for a specific  graph are the feasible solutions and the reconfiguration rule  is to change the color on exactly one vertex.  The reconfiguration graph in this case is called the coloring graph and it naturally arises in 
theoretical physics when studying the Glauber dynamics of an anti-ferromagnetic Potts model at zero temperature 
\cite{dyer2006randomly, jerrum1995very, molloy2004glauber, vigoda2001improved}. Structural properties, such as 
when this graph  is connected or has a Hamilton cycle have been considered; see  for example, 
\cite{ cereceda2008connectedness,  choo2011gray}. In Beier et al.~\cite{Beier2016ClassifyingCG} the question of which graphs can be coloring graphs is addressed.
Variations of the coloring graph have been considered as well, including restricting to  only non-isomorphic colorings \cite{haas2011canonical},  or using a different coloring rule such as Kempe-chains \cite{mohar2007kempe}.

Several recent survey papers cover different aspects of reconfiguration. The   paper by  Nishimura \cite{Nishimura-survey} summarizes the state of understanding of algorithmic and complexity questions for a wide range of reconfiguration settings. 
A good overview of reconfiguration for   graph coloring problems  and dominating sets problems is given in  the recent paper of Mynhardt
and  Nasserasr
\cite{Mynhardt-survey}.

In this paper we begin the  study of reconfiguration for zero forcing sets.  Zero forcing is a coloring process on a graph that has seen much  recent attention  (see \cite{HLA2} and the references therein) in part because of its connections to linear algebra \cite{AIM},  power domination \cite{reuf15, Deanetal}, and control of quantum systems \cite{graphinfect}. The color change rule is: A blue vertex $u$ can change the color of a white vertex $w$ to blue if $w$ is the unique white neighbor of $u$.  Given a graph $G$, a subset of vertices $B\subseteq V(G)$ is a zero forcing set if when  $B$ is the initial set of blue vertices and the coloring rule is applied repeatedly,  all vertices are eventually colored blue.  The zero forcing number, denoted by $\Z(G)$, is the minimum of $|B|$ over all zero forcing sets $B \subseteq V(G)$. A  zero forcing set $B\subseteq V(G)$ such that $|B|=\Z(G)$ is a {\em minimum zero forcing set}.

We study reconfiguration among minimum zero forcing sets using a token jumping reconfiguration rule, that is, we move directly between a minimum zero forcing set $B$ to minimum forcing set   $B'$ if the symmetric difference of $B$ and  $B'$ contains exactly two elements.  The reconfiguration graph in this context will be called the {\em zero forcing graph} and denoted by $\zfg(G)$; the graph $G$ is called the {\em base graph}. 
Specifically,  given a base graph $G$, we define the  zero forcing graph $\zfg(G)$ of $G$ to be the graph whose vertices are the minimum zero forcing sets of $G$ with an edge between vertices $u$ and $v$ if and only if $u$ can be obtained from $v$ by changing a single vertex.

This paper primarily addresses the structural properties of the zero forcing graph. In Section \ref{s:disc}, we construct infinite families of graphs with disconnected zero forcing graphs. In Section \ref{s:spec}, we show that paths, cycles and complete graphs can be zero forcing graphs, but that stars can not. 
In Section \ref{s:forest} we show that the zero forcing graph of every forest is connected, and that $\zfg(T)$ is $\left\{C_3,C_4\right\}$-free if and only if $T$ is a path. 
While this paper briefly considers the complexity of computing  the zero forcing graph, we do not address the fundamental question  of the complexity of determining whether  two zero forcing sets are in the same connected component of $\zfg(G)$, other than when $G$ is a forest.

%
%
%
%
%
%
%
%
%

\section{Preliminaries}\label{s:intro}

In this section we  present some initial examples of the zero forcing graph for several families of base graphs and  then develop some useful tools based on previous results on zero forcing.  

We first review   some notation; most concepts and notation are standard and can be found in \cite{Diestel} or other common books on graph theory. All graphs are simple, undirected, and finite.   The minimum and maximum degrees of vertices in $G$ are denoted by $\delta(G)$ and $\Delta(G)$, respectively. The maximum order of a clique in $G$ is denote by $\omega(G)$.  The union of disjoint  sets $X$ and $Y$ is denoted by $X\sqcup Y$.  For any two  graphs $G_1$ and $G_2$, the disjoint union of $G_1$ and $G_2$ can be denoted by $G_1 \sqcup G_2$ (if necessary by renaming the vertices so that $V(G_1)$ and $V(G_2)$ are disjoint sets). A {\em leaf} is a vertex of degree one.  The {\em Cartesian product} of $G=(V,E)$ and $G'=(V',E')$, denoted by $G\cp G'$, is the graph with vertex set $V\x V'$ and  $(v,v')$ and $(u,u')$ are adjacent if and only if 
 $v=u$ and $\{v',u'\}\in E'$, or
 $v'=u'$ and $\{v,u\}\in E$.


\subsection{Zero forcing graphs of some common families of graphs}

\begin{rem}\label{r:zfg-path} Since the zero forcing number of a path is one and a path of order $n\ge 2$ has exactly two zero forcing sets (its endpoints), $\zfg(P_n)=K_2$. 
\end{rem}

\begin{prop}\label{p:cyc}
For all $n \geq 3$, $\zfg(C_n) = C_n$.
\end{prop}

\begin{proof}
Let the vertices of $C_n$ be labeled $v_0, \dots, v_{n-1}$, with $v_i$ adjacent to $v_{i-1}$ and $v_{i+1}$ for each $0 \leq i \leq n-1$ and perform index arithmetic mod $n$. Then the minimum zero forcing sets are $B_i = \left\{v_i, v_{i+1}\right\}$ for each $0 \leq i \leq n-1$, and $B_i$ is adjacent to only $B_{i-1}$ and $B_{i+1}$ in $\zfg(C_n)$, so $\zfg(C_n) = C_n$.
\end{proof}

\begin{prop}\label{p:Kn}
For all $n \geq 1$, $\zfg(K_n) = K_n$.
\end{prop}
\begin{proof}
The minimum zero forcing sets of $K_n$ are all possible subsets of $n-1$ vertices. There are $n$ of these subsets, and they are all connected in $\zfg(K_n)$, so $\zfg(K_n) = K_n$.
\end{proof}

Observe that cycles and complete graphs have the unusual property that $G = \zfg(G)$ (we say unusual because these are the only graphs that have this property as far as we know).

\begin{prop}\label{p:K1n}
For all $n \geq 1$, $\zfg(K_{1,n-1}) = K_{n-1}$.
\end{prop}

\begin{proof}
The minimum zero forcing sets of $K_{1,n-1}$ are all possible subsets of $n-2$ leaf vertices. There are $n-1$ of these subsets, and they are all connected in $\zfg(K_{1,n-1})$, so $\zfg(K_{1,n-1}) = K_{n-1}$.
\end{proof}

\subsection{Tools from zero forcing and some implications for $\zfg (G)$}

Let $B$ be a  zero forcing set of a graph $G$.  
A {\em chronological list of forces} of $B$ is an ordered list  of  forces which when performed one at a time, color all vertices of the graph blue starting with the vertices in $B$ blue.  A {\em set of forces} is  the (unordered) set of all forces in a chronological list of forces.  Note that a chronological list of forces carries more information and several chronological lists of forces may have the same set of forces.   Furthermore, $B$ may have several sets of forces.  However, for any subset of vertices, the set of vertices that can be colored blue by repeated application of the color change rule does not depend on the forces chosen or their order.
Given a set of forces, a {\em forcing chain}  is a maximal  sequence of vertices $(v_1,v_2,\dots,v_k)$ such that for $i=1,\dots,k-1$,  $v_i \to v_{i+1}$. 
 A {\em reversal}  of $B$ is the set of last vertices of the zero forcing chains of a chronological list of forces of $B$, i.e., the set of vertices that do not perform forces.  Note that a given minimum zero forcing set may have more than one reversal  (because it may have more than one set of forces), and that every reversal of a zero forcing set is also a zero forcing set \cite{BBFHHVS}.

\begin{thm}\label{t:exclude}{\rm\cite{BBFHHVS}} (Exclusion Theorem) If $G$ is a connected graph of order two or more, then for every vertex $v$ of $G$ there exists a minimum zero forcing set that does not include  $v$. That is,
\[\cap\{B:B \mbox{ is a minimum zero forcing set}\}=\emptyset.\]
\end{thm}

\begin{prop}\label{p:orderZplus1}
Let $G$ be a connected graph of order $n\ge 2$ or more such that $\zfg(G)$ is connected.  Then $|V(\zfg(G))|\ge \Z(G)+1$. 
\end{prop} 
\bpf Let $|V(G)|=n$, $\Z(G)=z$, and $|V(\zfg(G))|=y$.  Choose a minimum zero forcing set $B_1$ for $G$. 
Since  $\zfg(G)$ is connected, we can construct it from $B_1$ as a sequence of subgraphs $G_1=(B_1,\emptyset),\dots,G_y=\zfg(G)$ by adding one zero forcing set at each step with each subgraph being connected. Let $B_k$ be the  vertex of $\zfg(G)$ added at the $k$th step (to create $G_k$). 
By Theorem \ref{t:exclude}, every vertex of $G$ is excluded by some minimum zero forcing set.  Observe that $B_1$ excludes $n-z$ vertices of $G$. When a new vertex $B_k\in V(\zfg(G))$ is added, it excludes at most one vertex of $G$ that was in its neighbor zero forcing set in $G_{k-1}$, and thus excludes at most one vertex of $G$ that has not already been excluded by some $B_i, i=1,\dots,k-1$.  Thus the sets $B_1,\dots,B_y$ exclude at most $n-z+(y-1)$ vertices.  Since $n-z+(y-1)\ge n$, $y\ge z+1$.  
\epf

The assumption that $\zfg(G)$ is connected is necessary, as seen in Example \ref{e:zfg-discon-large}.

\begin{prop}\label{disjoint_union}
For any graphs $G_1$ and $G_2$, $\zfg(G_1 \sqcup G_2) = \zfg(G_1) \cp \zfg(G_2)$.
\end{prop}

\begin{proof}
Every vertex in $\zfg(G_1 \sqcup G_2)$ is of the form $X \sqcup Y$, where $X$ is a vertex of $\zfg(G_1)$ and $Y$ is a vertex of $\zfg(G_2)$. Two vertices $X_1 \cup Y_1$ and $X_2 \cup Y_2$ are adjacent in $\zfg(G_1 \sqcup G_2)$ if and only if either $X_1$ and $X_2$ are adjacent in $\zfg(G_1)$ and $Y_1 = Y_2$ or $Y_1$ and $Y_2$ are adjacent in $\zfg(G_2)$ and $X_1 = X_2$. Thus $\zfg(G_1 \sqcup G_2) = \zfg(G_1) \cp \zfg(G_2)$.
\end{proof}


\begin{cor}\label{disjoint_connected}
If $G_1$ and $G_2$ are graphs with $\zfg(G_1)$ and $\zfg(G_2)$ connected, then $\zfg(G_1 \sqcup G_2)$ is connected.
\end{cor}


The observation about chronological lists of forces in the next remark is well known (and might be described as folklore).
\begin{rem}\label{p:delta}(Neighbor Trading)
Let $G$ be a graph. Let $B$ be a minimum zero forcing set of $G$ and $v\in B$ with $\deg_G(v)\ge 2$. Suppose there is chronological list of forces $\clf$ such that $v\to w$ is the first force performed.  Let $u\in N_G(v)$ and $u\ne w$.  Then $u\in B$ and $B\setminus \{u\}\cup \{w\}$ is a  minimum zero forcing set with chronological list of forces obtained from $\clf$ by replacing $v\to w$ by $v\to u$.  Observe that $B$ and  $B\setminus \{u\}\cup \{w\}$ are adjacent in $\zfg(G)$. 
Thus $\delta(G)\le \omega(\zfg(G))$. 
 \end{rem}

\begin{prop}\label{p:DeltaZ}
Let $G$ be a graph such that $\zfg(G)$ does not have a $K_3$ subgraph. Then $\Delta(\zfg(G))\le Z(G)$.  
 \end{prop}
\bpf Let $z=\Z(G)$ and let $B=\{x_1,\dots,x_z\}\in V(\zfg(G))$ such that $\deg_{\zfg(G)}(B)=\Delta(\zfg(G))$.  In order to be a neighbor of $B$, a vertex of $\zfg(G)$ must differ by exactly one vertex from $B$.    If $S\subsetneq  B$ and there exist distinct vertices $a,b\not \in B$ such  that $S\cup\{a\}$ and $S\cup\{b\}$ are zero forcing sets, then \{$B$, $S\cup\{a\}$,  $S\cup\{b\}$\} induces a $K_3$ in $\zfg(G)$.  Thus each subset of $B$  of order $z-1$ appears in at most one minimum zero forcing set other than $B$.  Since there are exactly $z$ subsets of $B$ having $z-1$ vertices, there are at most $z$ minimum zero forcing sets  with a symmetric difference of two from $B$.  Thus $\Delta(\zfg(G))\le Z(G)$.
\epf

The hypothesis of no $K_3$ in $\zfg(G)$ is necessary in Proposition \ref{p:DeltaZ}, since $\Z(K_3\circ K_1)=2$ and  $\Delta(\zfg(K_3\circ K_1))=6$, where $K_3\circ K_1$ is constructed by adding a leaf to each vertex of $K_3$ (the corona of $K_3$ with $K_1$).

\section{Disconnected zero forcing graphs}\label{s:disc}

A fundamental question in reconfiguration is: for two particular solutions to a problem, can one reconfigure between them? 
Or more generally, is it true that one can reconfigure between any pair of solutions to a particular problem? That is, is the reconfiguration graph connected? 
In this section we exhibit families of graphs for which the zero forcing graph is disconnected. In particular, we show that for many base graphs $G$, 
judiciously appending leaves to certain vertices of $G$ results in a graph $G'$ such that  $\zfg(G')$ is disconnected.  We start with an example that is generalized in Proposition \ref{p:add-leaves-discon}.

\begin{ex}\label{e:zfg-discon-large} For $r\ge 3$ let $G_r$ be the graph obtained from $K_r\cp P_3$ by adding a leaf to each vertex in each of  the two end copies of $K_r$. Figure \ref{fig:G3} shows $G_3$.  The only minimum zero forcing sets are the two sets of $r$  vertices of degree one at each end  ($\{u,v,w\}$ and $\{x,y,z\}$  in Figure \ref{fig:G3}).  Observe that $|V(\zfg(G))|=2$ and $ \Z(G)=r$. \vspace{-8pt}
\end{ex}
  \begin{figure}[!h]
\begin{center}
\scalebox{.5}{\includegraphics{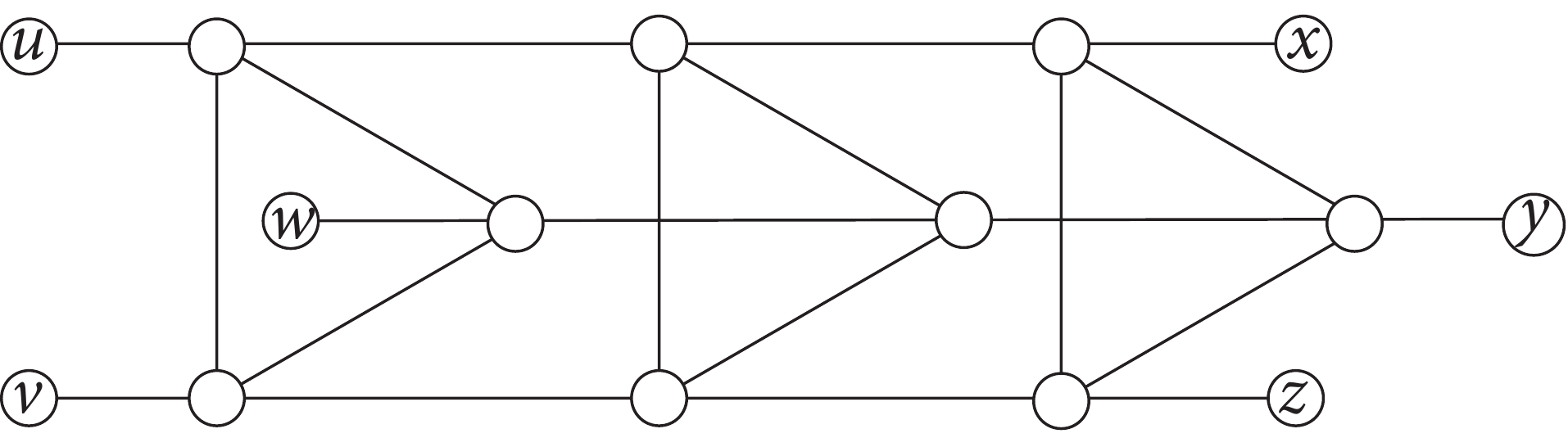}}
\caption{The graph $G_3$.}\label{fig:G3}\vspace{-15pt}
\end{center}
\end{figure}

The next two results give general operations that we can use to build graphs that have disconnected zero forcing graphs.

\begin{prop}\label{p:add-leaves-discon}
Suppose that $G$ is a graph with $\Z(G) = r$ for $r \geq 2$. Let $B$ and $B'$ be two disjoint minimum zero forcing sets of $G$ that are reversals of each other and for which there is no minimum zero forcing set in $G$ that intersects both $B$ and $B'$. Let $\wt G$ be the graph of order $|V(G)|+2r$ obtained by adding $r$ leaves $a_1, \dots, a_r$ to $B$ and $r$ leaves $a'_1, \dots, a'_r$ to $B'$ with one leaf adjacent to each vertex of $B$ and $B'$. Then $\Z(\wt G) = r$ and $\zfg(\wt G)$ is disconnected.
\end{prop}

\begin{proof}
First note that adding leaves to a graph does not decrease the zero forcing number, so $\Z(\wt G) \geq r$. Moreover, $\left\{a_1, \dots, a_r \right\}$ and $\left\{a'_1, \dots, a'_r \right\}$ are both zero forcing sets of $\wt G$, so $\Z(\wt G) = r$. Leaves must be endpoints of zero forcing chains, and there are $2r$ leaves and $r$ zero forcing chains for any minimum zero forcing set $B$ of $\wt G$, so every endpoint of the zero forcing chains must be a leaf. In particular, the minimum zero forcing set $B$ must consist only of vertices from $a_1, \dots, a_r$ and $a'_1, \dots, a'_r$. 
Observe that if $\wt B$ is a zero forcing set of $\wt G$, then the set of neighbors of vertices in $\wt B$ is a zero forcing set of $G$.  Thus the only minimum zero forcing sets are  $\wt B = \left\{a_1, \dots, a_r \right\}$ and $\wt B' = \left\{a'_1, \dots, a'_r \right\}$, or else there is a minimum zero forcing set of $G$ that intersects both $B$ and $B'$.
\end{proof}

Let $B$ be a minimum zero forcing set, let $\clf$ be a chronological list of forces, let $B'$ be the associated reversal. The list $\clf'$ of forces obtained from $\clf$ by reversing the order of the list and reversing each force is called the \emph{reversal of $\clf$}.  Note that $\clf'$ is a chronological list of forces of $B'$, and using this list produces $B$ as the reversal of $B'$. If $\Z(G)<|V(G)|$, then the vertex of $B$ that performs the first force is the last vertex forced starting with $B'$ and using $\clf'$, and vice versa.

\begin{rem}  
 If $\delta(G) = \Z(G) = r \geq 1$,  and $B$ is  any minimum zero forcing set $B$ of $G$, then there is a vertex in $G$ of degree $r$ that is not in $B$: Since $\delta(G) = \Z(G)$,   $r < |V(G)|$ is immediate.   The vertex $v$   that is forced last starting with $B$ and using $\clf$ is not in $B$ (since it is forced) and must have degree $r$ because it performs the first force starting with $B'$ and using $\clf'$. \end{rem}

\begin{prop}\label{p:G+H-discon} 
Suppose that $G$ and $H$ are any graphs with $\delta(G) = \delta(H) = \Z(G) = \Z(H) = r$ for $r \geq 2$. Let $X$ be any graph of order $|V(G)|+|V(H)|+r$ obtained from $G \sqcup H$ by choosing a minimum zero forcing set $B_G = \left\{v_1, \dots, v_r\right\}$ of $G$ of size $r$ and a minimum zero forcing set $B_H = \left\{w_1, \dots, w_r\right\}$ of $H$ of size $r$, adding $r$ new vertices $u_1, \dots, u_r$, making $u_1, \dots, u_r$ a complete graph with $\binom{r}{2}$ edges, and adding $2r$ edges of the form $v_i u_i$ and $u_i w_i$ for $i = 1, \dots, r$. Then $\Z(X) = r$ and $\zfg(X)$ is disconnected.
\end{prop}

\begin{proof}
Since $\delta(X) = r$, any zero forcing set of $X$ must have size at least $r$. Let $\clf$ be a chronological list of forces of $B_G$ that produces a set $B_G'$ as its reversal in $G$.  Then $B_G'$ is still a zero forcing set in $X$: 
To begin, $B_G'$ colors all of $G$ blue by the reversal $\clf'$ of $\clf$ (because the only vertices of $G$ adjacent to the $u_i$ in $X$ are in $B_G$, and the vertices of $B_G$ do not perform forces).  Then the vertices of $B_G$ color the new vertices $u_1, \dots, u_r$ blue,  $u_1, \dots, u_r$ color the vertices of $B_H$ blue, and finally $B_H$ colors all of $H$ blue. Thus $\Z(X) = r$.  Similarly, any reversal of $B_H$ in $H$ is still a zero forcing set in $X$. In any minimum zero forcing set $R$ of $X$, the vertex that makes the first force must have degree $r$ and it must initially have $r-1$ blue neighbors, so the vertex that makes the first force in $R$ cannot be a new vertex, and it cannot be $v_i$ or $w_i$ for $i = 1, \dots, r$. It must be some vertex from $G$ (respectively, $H$) that is not in $B_G$ (respectively, $B_H$). But now  its $r-1$ blue neighbors must be  the other elements of $R$, and so are also from $G$ (respectively, $H$). Thus the minimum zero forcing sets in $X$ are disconnected in $\zfg(X)$.
\end{proof}

Note that the operations in Propositions \ref{p:add-leaves-discon} and \ref{p:G+H-discon} can be applied to sufficiently large cycles and other circulant graphs to create an assortment of base graphs with disconnected zero forcing graphs.


\section{Graphs having a specified zero forcing graph}\label{s:spec}

In this section we characterize graphs whose zero forcing graph is a complete graph or a path.   We also construct graphs  whose zero forcing graph is a cycle or a hypercube but we suspect that there are other graphs  with the same zero forcing graph in these cases. Note that the hypercube $Q_d$ contains an induced $K_{1, d}$, but we show for $n \geq 3$ that there are no graphs $G$ having $\zfg(G)=K_{1,r}$. Hence zero forcing graphs will not have a forbidden subgraph characterization.

\begin{prop}\label{p:zfg=K2}
For connected graphs $G$, $\zfg(G) = K_2$ if and only if $G = P_n$ for some $n \geq 2$. 
\end{prop}

\begin{proof}
That $\zfg(P_n)=K_2$ for $n\ge 2$ was noted in Remark \ref{r:zfg-path}. Suppose that $\zfg(G) = K_2$. 
Then $\Z(G)=1$ by Proposition \ref{p:orderZplus1}, 
so $G = P_n$ for some $n \geq 1$. However when $n = 1$, $\zfg(G) = K_1$, so we must have $n \geq 2$.
\end{proof}

\begin{prop}
For connected graphs $G$ and $r \geq 3$, $\zfg(G) = K_r$ if and only if $G = K_r$ or $G = K_{1, r}$. 
\end{prop}

\begin{proof}
It was established in Propositions \ref{p:Kn} and \ref{p:K1n} that $\zfg(G)=K_r$ for $G = K_r$ and $G = K_{1, r}$.  
Suppose that $\zfg(G) = K_r$. Let $P^{(1)}, \dots, P^{(z)}$ denote a set of zero forcing chains for particular minimum  zero forcing set of $G$. Recall that a reversal of a zero forcing set is also a zero forcing set, and that two minimum zero forcing sets are adjacent in $\zfg(G)$ if their symmetric difference has size two. Thus there can be at most one $i$ with $|V(P^{(i)})|\geq 2$.

Moreover there must be some $i$ with $|V(P^{(i)})| \geq 2$, or else the zero forcing set would not be minimum. Without loss of generality, we assume that $|V(P^{(1)})| \geq 2$, $|V(P^{(i)})| = 1$ for each $i > 1$, and $P^{(1)}$ has endpoints $a$ and $b$. We refer to the paths $P^{(i)}$ for $i > 1$ as {\em singletons}.

Suppose that some singleton $w$ is not adjacent to any vertex on $P^{(1)}$. Then there would be a zero forcing set for $G$ of size $z-1$ consisting of $a$ and all of the singletons except for $w$, which is a contradiction. Thus every singleton is adjacent to at least one vertex on $P^{(1)}$.

We break the argument into cases depending on whether there is a singleton adjacent to an endpoint of $P^{(1)}$. For the first case, suppose that there is a singleton $w$ adjacent to an endpoint of $P^{(1)}$, and without loss of generality let the endpoint be $a$. Then $|V(P^{(1)})| = 2$, or else by neighbor trading, $\zfg(G)$ would have two non-adjacent vertices: the minimum zero forcing set containing $b$ and all singletons; and the minimum zero forcing set containing $a$, the neighbor of $a$ on $P^{(1)}$, and all singletons besides $w$. Thus $\Z(G) = |V(G)|-1$, which implies that $G$ is complete \cite{AIM}. Since $\zfg(G) = K_r$ and $G$ is complete, we have $G = K_r$.

For the second case, suppose that $a$ and $b$ both have degree $1$. Let $c$ be the high degree vertex of $G$ vertex on $P^{(1)}$  that is closest to $a$, and let $w$ be a singleton that is adjacent to $c$. If there were any vertex on $P^{(1)}$ between $b$ and $c$, then $\zfg(G)$ would have two non-adjacent vertices: the minimum zero forcing set containing $b$ and all singletons; and the minimum zero forcing set containing $a$, the neighbor of $c$ on $P^{(1)}$ that is closest to $b$, and all singletons besides $w$. Thus there is no vertex on $P^{(1)}$ between $b$ and $c$, so $c$ is the only high degree vertex of $G$ on $P^{(1)}$. By an analogous argument, there is no vertex on $P^{(1)}$ between $a$ and $c$. Thus $P^{(1)}$ has $3$ vertices and every singleton is adjacent to the middle vertex. If any singletons $w$ and $z$ were adjacent to each other, then there would be a minimum zero forcing set containing $a$ and all of the singletons except for $w$, a contradiction. Thus $G$ is a star, and $G = K_{1,r}$ because $\zfg(G) = K_r$.
\end{proof}

\begin{cor}\label{knzf}
If $G$ is connected and $\zfg(G) = K_r$ for $r \geq 2$, then $\Z(G) = r-1$. 
\end{cor}

 \begin{cor}  $\zfg(G) = K_r$ for $r \geq 2$ if and only if   one connected component  of $G$ is either $K_r$ or $K_{1,r}$ and all the other connected components are isolated vertices.\end{cor} 

  For $n\ge 5$, let $C_n(2)$ denote a cycle with vertices $1,\dots,n$ and two additional edges $\{1,3\}$ and $\{3,n\}$; 
  $C_7(2)$ is shown in Figure \ref{f:Cn(2)}.  Note that $n\ge 5$ implies there are at least two vertices of degree two in $C_n(2)$.\vspace{-8pt}
  
\begin{figure}[h!]\begin{center}
\scalebox{.5}{\includegraphics{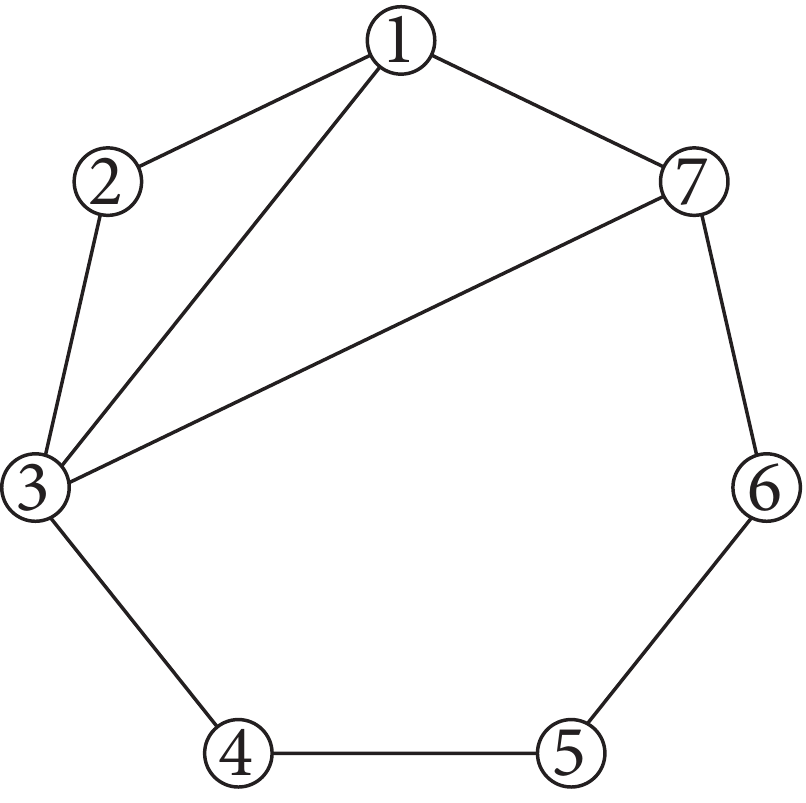}}
\vspace{-8pt}
\caption{\label{f:Cn(2)} The graph $C_7(2)$\vspace{-30pt}}
\end{center}\end{figure}

\begin{prop}\label{p:Cn+edges}
 For $n\ge 5$, $\zfg(C_n(2)) = P_{n-1}$.
\end{prop}
\bpf Observe that  $2=\delta(C_n(2))\le\Z(C_n(2))$, and $\{k,k+1\}$ is a zero forcing set for $k=1,\dots,n-1$.   In order for a set of two vertices to be a zero forcing set, one of the vertices must have degree 2 and the other vertex must be a neighbor of the degree-2 vertex. Thus these are the only minimum zero forcing sets  and $\zfg(C_n(2)) = P_{n-1}$.
\epf

We have defined $C_n(2)$ with $\deg(1)=\deg(n)=3$ and $\deg(3)=4$ for convenience in the proof, but for $n\ge 7$, there $\lc\frac n 2\rc-2$ nonisomorphic graphs $G$ having $\zfg(G)=P_{n-1}$. These other graphs can be constructed as a cycle $C_n$ with the two additional edges $\{1,k\}$ and $\{n,k\}$ with $3\le k\le \lc\frac n 2\rc$.  

Since $P_2=K_2$ is the zero forcing graph of a path, we have found graphs $G$ having $\zfg(G)=P_n$ for all $n\ne 3$.  Since $P_3=K_{1,2}$, it is a consequence of the next result 
that there does not exist a graph $G$ such that $\zfg(G)=P_3$.

 \begin{prop}\label{p:no-K1n} There does not exist a graph $G$ such that $\zfg(G)=K_{1, r}$ for $r \geq 2$.
\end{prop}
\bpf Suppose that there is a graph $G$ such that $\zfg(G)=K_{1,r}$ for $r \geq 2$. By Proposition \ref{disjoint_union}, $G$ has at most one connected component of size at least $2$, so without loss of generality we may assume that $G$ is connected. By Proposition \ref{p:DeltaZ}, $Z(G)\ge r$ and by Proposition \ref{p:orderZplus1}, $r+1\ge \Z(G)+1$, so $\Z(G)= r$. \newpage

Let $B = \{a_1, \dots, a_r\}$ be the vertex of $\zfg(G)$ of maximum degree.  Without loss of generality, the zero forcing sets of $G$ are of the form $B_0 = B$ and $B_i=B\cup \left\{b_i \right\} \setminus \left\{a_i\right\}$ for each $i = 1, \dots, r$; note that $b_i,i=1,\dots,r$ are distinct and each $b_i$ appears in exactly one minimum zero forcing set of $G$. 

Each zero forcing set has a reversal.  Without loss of generality, a reversal of $B_0$ is $B_1$, so there is a zero forcing process starting with $B_1$ and ending with $B_0$  
in which $a_2, \dots, a_r$ do not perform a force.  Using this process, $G$ has $r$ disjoint forcing paths, with one path $P$ starting at $b_1$ and ending at $a_1$ and the other paths consisting of the single vertices $a_2, \dots, a_r$.  
If $\deg_G(b_1) \geq 2$, then $b_1$ is in another zero forcing set by Remark \ref{p:delta} (neighbor trading), which is a contradiction.  

So assume $\deg_G(b_1)=1$.  Since $G$ is connected, there is some $i\ge 2$ such that $a_i$ is adjacent in $G$ to a vertex of $P$.  Starting at $b_1$ and proceeding along $P$ from $b_1$ to $a_1$ in forcing order, let $v$ be the first vertex on $P$ that is adjacent to some  $a_i$ with $i \geq 2$ and let $u$ be its successor on $P$. Then $B_1 \cup \{u\} \setminus\{a_i\}$ is also a zero forcing set, which is a contradiction. \epf

While a star cannot be a zero forcing graph, stars of any size arise as induced subgraphs, such as in a hypercube. For any $d$, the hypercube $Q_d$ is a zero forcing graph for multiple base graphs. Since $Q_d=K_2\cp K_2\cp \dots \cp K_2$, it follows from Proposition \ref{disjoint_union} that $\zfg(K_2\sqcup K_2\sqcup\dots\sqcup K_2)=Q_d$.  The tree exhibited in Proposition \ref{dcube} is another example.

Recall that $\zfg(C_n)=C_n$.  But the cycle is not the only graph that has its zero forcing graph equal to a cycle: Let $C_n+e$ denote a cycle with one edge added.  
\begin{prop}\label{p:Cn+edge}
For $n\ge 4$, $\zfg(C_n+e) = C_n$.   \end{prop}
\bpf Observe that  $2=\delta(C_n+e)\le\Z(C_n+e)$, and $\{1,n\}$ and $\{k,k+1\}$ for $k=1,\dots,n-1$ are all zero forcing sets.   In order for a set of two vertices to be zero forcing set, one of the vertices must have degree 2 and the other vertex must be a neighbor of the degree-2 vertex. Thus these are the only minimum zero forcing sets  and $\zfg(C_n+e) = C_{n}$.
\epf

The $H$-graph, shown in Figure \ref{f:H-graph}, is an acyclic graph that has the 4-cycle as its zero forcing graph. The minimum zero forcing sets for $H$ are $\{1,4\}, \{1,6\}, \{3,4\}$ and $\{3,6\}$.
\vspace{-3pt}
\begin{figure}[h!]\begin{center}
\scalebox{.5}{\includegraphics{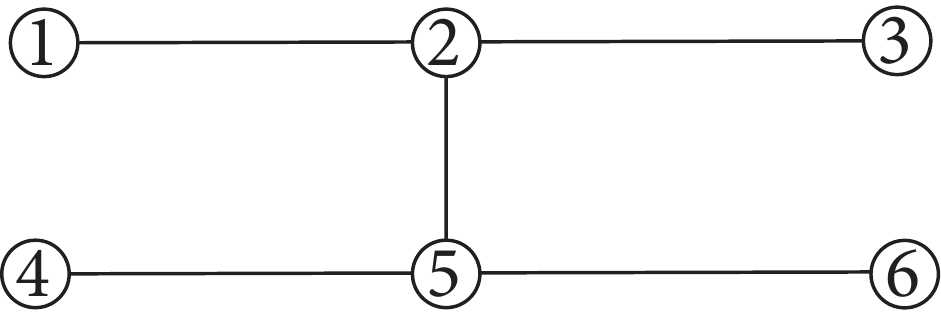}}
\vspace{-4pt}
\caption{\label{f:H-graph} The $H$-graph\vspace{-15pt}}
\end{center}\end{figure}


\section{Zero forcing graphs of forests}\label{s:forest}

In this section we show that the zero forcing graph of a forest is connected. We also show that the  zero forcing graph of a non-path tree must have  $C_3$ or $C_4$ as an induced subgraph, and construct trees that have hypercubes as their zero forcing graphs. The techniques developed for reconfiguration are then applied to prove  that every tree has a zero forcing set that consists entirely of leaves.\newpage

 A \emph{path cover} of a graph $G$ is a collection of induced paths in $G$ such that every vertex of $G$ 
is in exactly one path. The \emph{path cover number of $G$} is the minimum number of paths in a path cover of $G$, and is denoted by $\PCN(G)$.
Throughout this section, we use the following facts from \cite{AIM}: For a tree $T$, $\Z(T)=\PCN(T)$.  Every minimum path cover of a tree produces one or more minimum zero forcing sets by selecting any one endpoint from each path in the cover.  Every set of forcing chains of a zero forcing set of a graph is a path cover, and for a minimum zero forcing set of a  tree, the set of forcing chains  is a minimum path cover; the zero forcing set is a set of endpoints from this path cover.  Observe that each of these associations is usually many-to-one:  A zero forcing set often has more than one set of forcing chains, and a path cover that has $p$ paths of order two or more gives rise to $2^p$ zero forcing sets via endpoint choice.

\begin{lem}\label{induced_exterior}
If $T$ is a forest and $T'$ is a forest obtained by removing the vertices of a path $P$ in a minimum path cover of $T$, then $\zfg(T')$ is isomorphic to an induced subgraph of $\zfg(T)$.
\end{lem}

\begin{proof} 
Because $P$ is a path in a minimum path cover, the path cover  number of $T'$ is $\PCN(T')=\PCN(T)-1$, and thus $\Z(T')=\Z(T)-1$.  Fix  one endpoint $v$ of $P$.  If $B'$ is a minimum zero forcing set for $T'$, then $B'$ is a set of endpoints of a minimum path cover of $T'$,   $B' \cup \left\{v\right\}$ is a set of endpoints of a minimum path cover of $T$, and $B' \cup \left\{v\right\}$ is a minimum  zero forcing set of $T$.  If $B'_1$ and $B'_2$ are minimum zero forcing sets of $T'$, then $B'_1$ and $B'_2$ are adjacent in $\zfg(T')$ if and only if $B'_1 \cup \left\{v\right\}$ and $B'_2 \cup \left\{v\right\}$ are adjacent in $\zfg(T)$. Thus mapping $B'$ to $B' \cup \left\{v\right\}$  maps $\zfg(T')$ to an induced subgraph of $\zfg(T)$.
\end{proof}

We need some terminology.  If $G$ is a graph and $v\in V(G)$, then  the graph obtained from $G$ by removing $v$ is denoted by $G-v$.
 A \emph{high degree vertex} of a tree is a vertex of degree at least three. A tree is a {\em generalized star} if it has at most one high degree vertex; if a generalized star has a high degree vertex, this vertex is called the {\em center}.  A path $P$ of  a tree $T$ is a {\em pendent path} of vertex  $v$ if $P$ is a component of $T - v$ and  $P$ is connected in $T$ to $v$ by one of its end-points.  
A {\em pendent generalized star} of a tree $T$ is a connected induced subgraph $R$ of $T$ such that 
	there is exactly one high degree vertex $v$ of $T$ in $R$ ($v$ is called the {\em center} of $R$), all but one  of the components of $T - v$  are pendent paths (and one component is not a pendent path), and
	$R$ is induced by the vertices of the pendent path components of $T-v$  and $v$. 
It is established in  \cite{FH}  that any  tree  contains a pendent generalized star or is a generalized star.  

\begin{thm}\label{tree_con}
For every tree $T$, $\zfg(T)$ is connected.
\end{thm}

\begin{proof} 
We prove this by strong induction on the order of $T$. Trees of order $2$ or $3$ are paths. Since the zero forcing graph of a path is connected, let $T$ be a tree that is not a path, and suppose for the induction hypothesis that $\zfg(T')$ is connected for all trees $T'$ of order less than $T$. 

Since $T$ is not a path, there exists a vertex $u \in V(T)$ such  $T-u$ has at least three connected components, at most one of which is not a path. We label the vertices of two of these paths $v_1 \dots v_j$ and $w_1 \dots w_k$ for $j, k \geq 1$ where $\dist(v_i, u) = i$ for each $1 \leq i \leq j$ and $\dist(w_i, u) = i$ for each $1 \leq i \leq k$. 

Let $\PC = \left\{R_1, \dots, R_z\right\}$ 
be a minimum path cover of $T$, where $z = \Z(T)$. Let $R_i$ be the path in $\PC$ that contains the degree-$1$ vertex $v_j$. The path $R_i$ must contain at least $j$ vertices (since  $\PC $ is a minimum path cover).
 The path $R_i$ consists of $v_1 \dots v_j$, or $R_i$ consists of $v_j \dots v_1 u w_1 \dots w_k$, or $R_i$ consists of $v_j \dots v_1 u$ and additional vertices but not $w_1 \dots w_k$, in which case there is a path $R'' \in \PC$ that consists of $w_1 \dots w_k$. Thus we can group all minimum zero forcing sets into three classes, where overlap is allowed between classes: Class 1 consists of minimum zero forcing sets $B$ such that $B$ is the set of endpoints of some minimum path cover containing $v_1 \dots v_j$ as a path in the cover, Class 2 consists of minimum zero forcing sets $B$ such that $B$ the set of endpoints of some minimum path cover containing $w_1 \dots w_k$ as a path in the cover, and Class 3 consists of minimum zero forcing sets $B$ such that $B$ is the set of endpoints of some minimum path cover containing $v_j \dots v_1 u w_1 \dots w_k$ as a path in the cover.  Since $T-u$ has at least three connected components, each  of the classes is nonempty.

Every minimum zero forcing set is in at least one of  Class 1, 2, or 3. We first prove that Class 1 and Class 3 are each connected in $\zfg(T)$. 
Then we show that Class 1 intersects Class 3. 
The proofs that Class 2 is connected and that Class 2 intersects Class 3 are similar. 
This will suffice to show that $\zfg(T)$ is connected.

Let $T'$ be the tree obtained from $T$ by removing $v_1, \dots ,v_j$. Class 1 consists of minimum zero forcing sets of the form $\left\{v_1\right\} \cup B'$ or $\left\{v_j\right\} \cup B'$ where $B'$ is any minimum zero forcing set of $T'$. Since $\zfg(T')$ is connected by induction hypothesis, the subgraph of $\zfg(T)$ restricted to minimum zero forcing sets of the form $\left\{v_1\right\} \cup B'$ is also connected by Lemma \ref{induced_exterior}. Similarly, the subgraph of $\zfg(T)$ restricted to minimum zero forcing sets of the form $\left\{v_j\right\} \cup B'$ is also connected. Moreover, if we fix a particular minimum zero forcing set $B'$ of $T'$, then there is an edge between $\left\{v_1\right\} \cup B'$ and $\left\{v_j\right\} \cup B'$ in $\zfg(T)$ if $v_1 \neq v_j$, so the subgraph of $\zfg(T)$ restricted to Class 1 is connected. 

Let $T''$ be the graph obtained from $T$ by removing $v_j \dots v_1 u w_1 \dots w_k$ ($T''$ is not necessarily connected). Class 3 consists of minimum zero forcing sets of the form $\left\{v_j\right\} \cup B''$ or $\left\{w_k\right\} \cup B''$ where $B''$ is any minimum zero forcing set of $T''$. For each connected component $X$ of $T''$, $\zfg(X)$ is connected by induction hypothesis, so $\zfg(T'')$ is also connected by Corollary \ref{disjoint_connected}. Thus the subgraph of $\zfg(T)$ restricted to minimum zero forcing sets of the form $\left\{v_j\right\} \cup B''$ is also connected. Similarly, the subgraph of $\zfg(T)$ restricted to minimum zero forcing sets of the form $\left\{w_k\right\} \cup B''$ is also connected. Moreover, if we fix a minimum zero forcing set $B''$ of $T''$, then there is an edge between $\left\{v_j\right\} \cup B''$ and $\left\{w_k\right\} \cup B''$ in $\zfg(T)$, so the subgraph of $\zfg(T)$ restricted to Class 3 is connected. 

Now we prove that Class 1 intersects Class 3. We consider a minimum path cover $\PC = \left\{R_1, \dots, R_z\right\}$ of $T$ where $R_1 = v_1 \dots v_j$.  If $ w_1 \dots w_k$ is a path in  $\PC$, then there is another path that includes $u$.   
If the  path that includes $u$ does not include $w_1$, then it must have the form $y_\ell\dots y_1 u y'_1\dots y'_{\ell'}$.  We can replace the paths $ w_1 \dots w_k$ and $y_\ell\dots y_1 u y'_1\dots y'_{\ell'}$ by $ w_k \dots w_1 u y_1\dots y_\ell$ and $y'_1\dots y'_{\ell'}$. 
So without loss of generality, assume $R_2$ contains $w_k \dots w_1 u$ and some number of other vertices including $q$, which is the farthest vertex in $R_2$ from $w_k$. Let $\PC' =  \left\{R'_1, \dots, R'_z\right\}$ be obtained from $\PC$ by setting $R'_i = R_i$ for each $i \geq 3$, deleting $w_k \dots w_1 u$ from $R_2$ to form the new path $R'_2$, and adding $w_k \dots w_1 u$ to $R_1$ to form the new path $R'_1 = v_j \dots v_1 u w_1 \dots w_k$. Let $B_0$ be a minimum zero forcing set obtained from $\PC$ by selecting a single endpoint from each of $R_3, \dots, R_z$, and selecting the endpoints $v_j$ and $q$ from $R_1$ and $R_2$ respectively. Then $B_0$ is also a minimum zero forcing set that corresponds to $\PC'$, since each vertex in $B_0$ is an endpoint of a distinct path in $\PC'$. Thus $B_0$ is in both Class 1 and Class 3.  
\end{proof}

The next corollary follows from Corollary \ref{disjoint_connected} and Theorem \ref{tree_con}.

\begin{cor}\label{forest-con}
For every forest $T$, $\zfg(T)$ is connected.
\end{cor}

Although $\zfg(G)$ is always connected for base graphs $G$ with no cycles, this is not true in general if $G$ contains a cycle, as can be seen from Example \ref{e:zfg-discon-large} and the constructions of Section \ref{s:disc}.  Given this pivotal role of cycles in the base graph, it is interesting that zero forcing graphs of acyclic base graphs that are not paths have cycles, as established in the next result.


\begin{thm}\label{c3c4}
For every tree $T$, $\zfg(T)$ is $\left\{C_3,C_4\right\}$-free if and only if $T$ is a path.
\end{thm}

\begin{proof}
If $T$ is a path, then the statement is clearly true, so we prove the other direction. We will prove by strong induction on the order of $T$ that $\zfg(T)$ has a $C_3$ or $C_4$ subgraph if $T$ is a non-path tree. 
Suppose that $T$ is a non-path tree. 
 If $T$ is a generalized star with at least three leaves, then the subgraph of $\zfg(T)$ restricted to minimum zero forcing sets that only include leaves is a complete graph of order at least three, so $\zfg(T)$ contains a triangle; this covers the base case of order four.    Note the  $H$-graph, shown in Figure \ref{f:H-graph} is a tree with $\zfg(T) = C_4$.
Suppose as the induction hypothesis that $\zfg(T')$ contains $C_3$ or $C_4$ for all non-path trees $T'$ of order less than $T$.

Now we assume that $T$ is not a generalized star. As in the last proof, we use the fact that for any tree $T$ that is not a path, there exists a vertex $u \in V(T)$ such that  $T-u$  has at most one non-path component and least two connected components that are paths. We label these two paths $v_1 \dots v_j$ and $w_1 \dots w_k$ for $j, k \geq 1$ where $\dist(v_i, u) = i$ for each $1 \leq i \leq j$ and $\dist(w_i, u) = i$ for each $1 \leq i \leq k$. 
We choose an arbitrary minimum path cover of $T$, and consider two cases. 

If there is a path with both $v_j$ and $w_k$ in the minimum path cover, then let $T'$ be the forest obtained from removing that path. We have that $\zfg(T')$ is connected by Corollary \ref{forest-con}. 
Suppose first that all connected components of $T'$ are paths.  Since $T$ is not a generalized star,  at least one of the components is a path  $P$ that is not pendent, and thus has at least two vertices. Denote the two distinct endpoints of $P$ by $x$ and $y$. Then there is a subset $S \subset V(T)$ of path endpoints and four vertices of $\zfg(T)$ of the form $S \cup \left\{v_j, x\right\}, S \cup \left\{v_j, y\right\}, S \cup \left\{w_k, x\right\}, S \cup \left\{w_k, y\right\}$ that form a $C_4$ in $\zfg(T)$. So now assume $T'$ has a connected component $C$ that is not a path. 
Then by the induction hypothesis $\zfg(C)$ contains $C_3$ or $C_4$. This implies that $\zfg(T)$ also contains $C_3$ or $C_4$, since $\zfg(T)$ has a subgraph isomorphic to $\zfg(C)$ by Lemma \ref{induced_exterior}.

For the second case, suppose that there is not a path in the minimum path cover with both $v_j$ and $w_k$. Then the minimum path cover has the path $v_1 \dots v_j$ or the path $w_1 \dots w_k$. Without loss of generality, suppose that the minimum path cover has the path $v_1 \dots v_j$. Let $T'$ be the subtree of $T$ obtained by removing $v_1 \dots v_j$. Then $T'$ is not a path because $T$ is not a generalized star, and $\zfg(T)$ has a subgraph isomorphic to $\zfg(T')$, so $\zfg(T)$ has a $C_3$ or $C_4$ by induction hypothesis.
\end{proof}

\begin{cor}
For every forest $T$, $\zfg(T)$ is $\left\{C_3,C_4\right\}$-free if and only if $T$ is a disjoint union of one or more paths with at most one path having order greater than one. 
\end{cor}

Based on Theorem \ref{c3c4}, it is natural to ask what other induced cycles can occur in $\zfg(T)$ for trees $T$. The next result shows that there is no bound on the lengths of induced cycles that occur in $\zfg(T)$ for trees $T$, and that any even length at least $4$ is possible.

\begin{prop}\label{dcube}
For each $d > 0$, there exists a tree $T$ such that $\zfg(T)$ is the $d$-dimensional hypercube.
\end{prop} 

\begin{proof}
Let $G_d$ be the tree obtained from $P_d$ by adding two leaves to every vertex. Then $\Z(G_d) = d$ and the only minimum path cover of $G_d$ is the path cover where every path consists of a vertex from $P_d$ and its two leaves. Thus $\zfg(G_d)$ is the $d$-dimensional hypercube.
\end{proof}

The $d$-dimensional hypercube contains an induced $2d$-cycle for each $d \geq 2$: For example, consider the $2d$ binary strings of length $d$ consisting of the $d$ binary strings with $d-1$ zeroes and a single one, the $d-1$ binary strings with $d-2$ zeroes and $2$ adjacent ones, and the string with $d-2$ zeroes and ones at the beginning and end. These $2d$ strings form an induced $2d$-cycle in the $d$-dimensional hypercube.  

\begin{cor}
For each even $d \geq 4$, there exists a tree $T$ such that $\zfg(T)$ contains an induced cycle of length $d$.
\end{cor}

Proposition \ref{dcube} also shows that there is no bound on the number of induced $4$-cycles that can occur in $\zfg(T)$ for a tree $T$, even when $\zfg(T)$ has no $3$-cycles. 

Another natural problem is to determine which trees have the same zero forcing graphs. The next result provides a partial answer.

\begin{prop}
Suppose that $T$ is a tree with vertices $u, w$ each of degree $2$ that have a common neighbor $v$ also of degree $2$. Then $T$ has the same minimum zero forcing sets as the tree $T'$ obtained by removing $v$ and replacing it with an edge between $u$ and $w$, so $\zfg(T) = \zfg(T')$.
\end{prop}

\begin{proof}
Any minimum zero forcing set of $T$ (respectively,  $T'$) can be obtained from choosing a single endpoint from each path in a minimum path cover of $T$ (respectively,  $T'$). No minimum path cover of $T$ has a path with endpoint $v$, since otherwise that path could be combined with any path whose endpoint is adjacent to $v$ to make a path cover with fewer paths. Thus every minimum path cover of $T$ can be transformed into a path cover of $T'$ with the same endpoints. 

Moreover no minimum path cover of $T'$ has both a path with endpoint $u$ and a different path with endpoint $w$, or else these paths could be combined to make a path cover with fewer paths. Thus every minimum path cover of $T'$ can be transformed into a path cover of $T$ with the same endpoints. This implies that $T$ and $T'$ have the same minimum zero forcing sets.
\end{proof}

The last result is not true in general for non-tree graphs if there is a cycle that contains $u, v, w$, since two zero forcing paths that cover a cycle cannot be merged in general to make one zero forcing path that covers a cycle. For example, the result would fail to hold for $C_n$, because $\zfg(C_n)=C_n$.

It is often the case that the study of reconfiguration leads to techniques that are useful in the study of the original problem.  We  now apply exchange ideas developed here in the poof of Theorem \ref{t:leaf-zfs}.  An algorithm is presented in \cite{FH} in which the path cover number of a tree $T$ is computed by identifying a pendent generalized star $R$, covering $R$ with paths, deleting the vertices of $R$ from $T$, and repeating until an empty graph or generalized star is obtained.  In other words, a minimum path cover  of $T$ can be found by choosing a pendent generalized star $R$, removing it to obtain a tree $T'$, and separately covering $T'$ and $R$.  
This  is the 
approach  taken in the next proof.

\begin{thm}\label{t:leaf-zfs}
For any tree $T$, there exists a minimum path cover of $T$ in which every path has a leaf of $T$.
\end{thm}

\begin{proof}
The proof is by strong induction on the order of $T$.  The theorem is clearly true for  generalized stars (including paths), so we suppose that $T$ is not a generalized star and that every tree $T'$ of order less than $T$ has a minimum path cover  in which every path has a leaf of $T'$.  
Choose a pendent generalized star $R$ of $T$ with center $u$, denote the path components of $T-u$ by $R_1,R_2,\dots,R_k$,  let $T'$ be the  component of $T-u$ that is not a pendent path, and let $w$ be the neighbor of $u$ in $T$ that is in   $T'$.  By the induction hypothesis, there is a minimum path cover $\PC=\{P^{(1)},\dots,P^{(p)}\}$ of $T'$ in which every path has a leaf of $T'$.  Without loss of generality, $P^{(1)}$ is the path in this cover that contains $w$.  If $w$ is not an endpoint of $P^{(1)}$, then $P^{(1)}$ has a leaf endpoint in $T$.  Thus $\{P^{(1)},\dots,P^{(p)}\}\cup\{R_1uR_2,R_3,\dots,R_k\}$ is a minimum path cover of $T$ in which  every path has a leaf of $T$.  
 Now suppose $w$ is an endpoint of $P^{(1)}$, so $P^{(1)}=Uw$ where $U$ is a nonempty subpath of $P^{(1)}$.  Then $\{P^{(2)},\dots,P^{(p)}\}\cup\{U w u R_1,R_2, R_3,\dots,R_k\}$ is a minimum path cover of $T$ in which  every path has a leaf of $T$.  
\end{proof}

Since  a path cover of a forest is the union of path covers of the trees that are its connected components, Theorem \ref{t:leaf-zfs} extends to forests.  

\begin{cor}\label{zfsleaf}
For any forest $T$, $\zfg(T)$ contains a vertex consisting entirely of leaves, i.e. $T$ has a minimum zero forcing set in which each vertex is a  leaf.
\end{cor}


\section{Computational results}

In this section, we prove several computational results about minimum zero forcing sets and path covers by using two lemmas. The first lemma shows that we can check in polynomial time in the order of a given graph whether a given set of vertices is a zero forcing set of the graph. It uses Algorithm \ref{zfscheck} to determine whether a given set of vertices is a zero forcing set.

\begin{algorithm}
\caption{Check whether $S_0$ is a zero forcing set of a graph $G$}
\label{zfscheck}
\begin{algorithmic}
\STATE Initialize $S=S_0$.
\WHILE{ $S \neq V(G)$}
\FOR{$v \in S$}
\IF{$v$ has exactly $1$ neighbor $w$ in $V(G)-S$}
\STATE Add $w$ to $S$.
\STATE Exit for-loop and go back to the while-statement.
\ENDIF
\ENDFOR
\STATE Exit while-loop, return {not a zero forcing set} (since  $S \neq V(G)$ and no vertex can force).
\ENDWHILE
\STATE Return {is a zero forcing set} (since $S = V(G)$).
\end{algorithmic}
\end{algorithm}

\begin{lem}\label{zfspoly}
For graphs $G$ of order $n$,  whether a given subset $S \subset V(G)$ is a zero forcing set of $G$ can be determined in $O(n^3)$ operations.
\end{lem}

\begin{proof}
Given the graph $G$ as a list of edges, run through the list in order to compute the set of neighbors $N(v)$ of each vertex $v$. This computation takes $O(n^2)$ operations since there are at most $\binom{n}{2}$ edges in $G$ and each edge updates the neighborhoods of two vertices. Now we use Algorithm \ref{zfscheck} to determine whether a given subset $S \subset V(G)$ is a zero forcing set of $G$. 

The while-loop runs at most $n$ times since at most $n$ elements can be added to the set $S$. The for-loop runs at most $n$ times since $S$ always has size at most $n$. Checking the if-condition in the interior of the for-loop takes $O(n)$ operations since we already calculated the neighborhood of each vertex. Thus checking if $S$ is a zero forcing set takes $O(n^3)$ operations.
\end{proof}

Lemma \ref{pgspoly} shows that in any non-path tree, we can find the center of a generalized star or a pendent generalized star of the tree in polynomial time in the order of the tree.  It uses Algorithm \ref{gpscheck} to find the center of a generalized star or a pendent generalized star of a non-path tree.  Note that the algorithm returns additional information that is used the proof of Theorem \ref{t:PC-leaf}.

\begin{lem}\label{pgspoly}
For non-path trees $T$ of order $n$,  the center of a generalized star or pendent generalized star of $T$ can be found in $O(n^2)$ operations.
\end{lem}

\begin{proof}
As in Lemma \ref{zfspoly}, we first compute the neighborhood $N(v)$ of all vertices $v \in T$ in $O(n^2)$ operations, recording both the neighborhoods and the degrees. Then we run Algorithm \ref{gpscheck} to find the center vertex of a generalized star or a pendent generalized star of $T$, assuming that $T$ is not a path. In  this algorithm, we check each vertex $v \in V(T)$ to see whether it is the center of a generalized star or a pendent generalized star until we find a center, after verifying that $T$ is not a path. We use a variable $k$ to count the number of non-path connected components of $T-v$, so we require that $k \leq 1$ in order for $v$ to be the center of a generalized star or a pendent generalized star.

Checking if there is a vertex in $T$ of degree at least $3$ takes $O(n)$ operations. Assuming there is, the outer for-loop runs at most $n$ times since $T$ has order $n$. The connected components of any graph $G = (V, E)$ can be found in $O(|V|+|E|)$ operations, so computing the connected components of $T-v$ takes $O(n)$ operations since $T$ has order $n$ and thus at most $n-1$ edges. We claim that the inner for-loop and its interior take at most $O(n)$ operations: Note that   the components $C$ partition the vertices of $T-v$.  For $w\in V(C)$,  $\deg_ {T-v}(w)=\deg_T(w) -1$ if $w\in N(v)$ and $\deg_ {T-v}(w)=\deg_T(w)$ if $w\not\in N(v)$; recall we recorded the degree and the neighborhood of each vertex in $T$. Thus finding the center of a generalized star or pendent generalized star of $T$ takes $O(n^2)$ operations.
\end{proof}

\begin{algorithm}
\begin{algorithmic} 
\caption{Find the center of a generalized star or a pendent generalized star of a non-path tree $T$}
\label{gpscheck}
\IF{every vertex in $T$ has degree at most $2$}
\STATE Return {$T$ does not have a generalized star or pendent generalized star.}
\ENDIF
\FOR{$v \in V(T)$ of degree at least $3$}
\STATE $k = 0$ (\# of non-pendent paths)
\STATE Compute the connected components $\C = \left\{C_1, \dots, C_j\right\}$ of $T-v$. 
\FOR{$X \in C$}
\IF{$X$ has a vertex of degree at least $3$ in $T$}
\STATE Increment $k$.
\STATE Record the component number of $X$ in $npp$.
\ENDIF
\ENDFOR

\IF{$k = 0$}
\STATE {Return $T$ has a generalized star with center $v$ and set $\C$ of components of $T-v$.}
\ENDIF

\IF{$k = 1$}
\STATE Return {$T$ has a pendent generalized star with center $v$, and set $\C$ of components of $T-v$, and number $npp$ of the component that is not a pendent path.}
\ENDIF
\ENDFOR 
\end{algorithmic}
\end{algorithm}
\begin{thm}\label{t:PC-leaf}
For trees $T$ of order $n$,  a minimum path cover of $T$ in which every path has a leaf of $T$ can be found  in $O(n^3)$ operations. 
\end{thm}

\begin{proof}
We can check if $T$ is a path in $O(n)$ operations. If it is, then  $T$  is its own path cover. For a non-path tree,  there exists a constant $d$ such that we can find the center vertex $v$ of a pendent generalized star or generalized star and the connected components of $T-v$ in at most $d n^2$ operations by Lemma \ref{pgspoly}. We prove by strong induction on $n$ that there exists a constant $c$ such that we can find a minimum path cover of $T$ in which every path has a leaf in at most $c n^3$ operations. A base case can be done in constant $c'$ time for some constant $c'$.  Let $c$ be a constant with $c > \max(c', d)$ for which it is possible to list the vertices of $G$ in at most $c n$ operations for any graph $G$ of order $n$.

Given any non-path tree $T$ of order $n$, we find the center vertex $v$ of a generalized star or a pendent generalized star of $T$ and the connected components of $T-v$ in at most $d n^2$ operations by Lemma \ref{pgspoly}.  If $v$ is returned as  the center of the generalized star $T$, then join the  first two  components with $v$ into one path. Take that path, plus each of the remaining components, to be the path cover.

So assume $v$ is returned as the center of a pendent generalized star of $T$.  As in the proof of Theorem \ref{t:leaf-zfs}, denote the pendent path components of $T-v$ by $R_1,R_2,\dots,R_k$, let $T'$ be the component of $T-v$ that is not a pendent path, and let $w$ be the neighbor of $v$ in $T$ that is in $T'$. Note that the algorithm in Lemma \ref{pgspoly} tells us which component of $T-v$ is not a pendent path. By the induction hypothesis, we can find a minimum path cover $\PC=\{P^{(1)},\dots,P^{(p)}\}$ of $T'$ in at most $c(n-1)^3$ operations in which every path has a leaf of $T'$. Let $P^{(1)}$ be the path in this cover that contains $w$.  If $w$ is not an endpoint of $P^{(1)}$, then $P^{(1)}$ has a leaf endpoint in $T$.  
Thus $\{P^{(1)},\dots,P^{(p)}\}\cup\{R_1vR_2,R_3,\dots,R_k\}$ is a minimum path cover of $T$ in which  every path has a leaf of $T$, and we computed this path cover in at most $d n^2 + c(n-1)^3 + c n \leq c n^3$ operations.
 If $w$ is an endpoint of $P^{(1)}$, then $P^{(1)}=Uw$ where $U$ is a nonempty subpath of $P^{(1)}$, so $\{P^{(2)},\dots,P^{(p)}\}\cup\{U w v R_1,R_2, R_3,\dots,R_k\}$ is a minimum path cover of $T$ in which  every path has a leaf of $T$.  Again we found this path cover in at most $d n^2 + c(n-1)^3 + c n \leq c n^3$ operations. Thus we can find a minimum path cover of $T$ in which every path has a leaf  of $T$ in $O(n^3)$ operations. 
\end{proof}

 Given a path cover in which each path contains a leaf of $T$, we can find  a minimum zero forcing set of $T$ consisting of all leaves by checking the degrees of the vertices of each path and choosing a leaf from each path.
 
\begin{cor}
For trees $T$ of order $n$, we can find a minimum zero forcing set of $T$ consisting of all leaves in $O(n^3)$ operations. 
\end{cor}

Let $\ZFS(G)$ denote the function which takes a graph $G$ as input and outputs the list of minimum zero forcing sets of $G$. The next result shows that computing $\ZFS(G)$ takes at worst $2^{\Theta(n)}$ operations for graphs $G$ of order $n$. In particular, there are trees $T$ of order $n$ for which computing $\ZFS(T)$ requires $2^{\Omega(n)}$ operations. We use the next algorithm.

\begin{algorithm}

\begin{algorithmic}
\caption{ Compute $\Z(G)$ and  $\ZFS(G)$}
\label{compzfs}
\STATE Let $k = 1$, $ZFS=\emptyset$, and {\tt Size = False}. 
\WHILE{{\tt Size = False}}
\FOR{every $S \subseteq V(G)$ of size $k$}
\STATE Determine whether $S$ is a zero forcing set of $G$. 
\IF{$S$ is a zero forcing set of $G$}
\STATE {{\tt Size = True}}.  (A zero forcing set is found so $\Z(G)=k$.) 
\STATE $ZFS = ZFS\cup\{S\}$.
\ENDIF
\ENDFOR
\IF{{\tt Size = True}}
\STATE Return $k$ and $ZFS$.
\ENDIF
\STATE $k=k+1$
\ENDWHILE
\end{algorithmic}
\end{algorithm}

\begin{prop}
For graphs $G$ of order $n$, computing $\ZFS(G)$ takes $2^{\Theta(n)}$ operations in the worst case.\label{exptime}
\end{prop}

\begin{proof}
First we prove the lower bound. Proposition \ref{dcube} shows that there are trees $G_n$ of order $3n$ for which $\zfg(T)$ has order $2^n$. Thus the list of minimum zero forcing sets of $G_n$ has length $2^n$, so it takes $\Omega(2^n)$ operations to output the list. 

For the upper bound, we use the brute-force Algorithm \ref{compzfs} to generate $\ZFS(G)$.
Algorithm  \ref{compzfs} takes $2^{O(n)}$ operations, since there are $2^n$ subsets of $V(G)$, and checking whether a subset $S$ is a zero forcing set takes polynomial time in $n$ by Lemma \ref{zfspoly}. Thus $\ZFS(G)$ can be computed in $2^{O(n)}$ operations.
\end{proof}

As a corollary, we note that computing the zero forcing graph for graphs of order $n$ also takes $2^{\Theta(n)}$ operations, since we can compute $\zfg(G)$ in $2^{O(n)}$ operations from $\ZFS(G)$. 

\begin{cor}
For graphs $G$ of order $n$, computing $\zfg(G)$ takes $2^{\Theta(n)}$ operations in the worst case.
\end{cor}

For trees $T$ of order $n$, we obtain the stronger result that computing all minimum path covers of $T$ takes $2^{\Theta(n)}$ operations in the worst case. This gives an alternative method for computing $\zfg(T)$ in $2^{O(n)}$ operations, since $\ZFS(T)$ can be computed from a list of all minimum path covers of $T$ in $2^{O(n)}$ operations. 

\begin{prop}
For trees $T$ of order $n$, computing a list of all minimum path covers of $T$ takes $2^{\Theta(n)}$ operations in the worst case. \label{exptimepath}
\end{prop}
\begin{proof}
First we prove the lower bound. Let $H_n$ be the graph obtained from $P_n$ by adding three leaves to every vertex. Any minimum path cover of $H_n$ has all paths of length $1$ or $3$, with the paths of length $1$ all added leaves and the paths of length $3$ all having a vertex from the path $P_n$ as the middle vertex and two added leaves as the outer vertices. So the tree $H_n$ has order $4n$ and $3^n$ minimum path covers. Thus it takes $\Omega(3^n)$ operations to output a list of minimum path covers for $H_n$.

For the upper bound, we combine judicious record-keeping with the proof method in Theorem \ref{tree_con}. The input is a tree $T$ of order $n$. We check if $T$ is a path in polynomial time in $n$. If so, we output $T$.

If $T$ is not a path, then we list all subtrees of $T$ in order of size. This takes $2^{O(n)}$ operations, since there are $2^n$ subsets $S$ of vertices of $T$, and checking whether $S$ is connected takes polynomial time in $n$. 

Going through the list in order, we compute a list of minimum path covers for each subtree $T'$ of $T$ and we store each list in memory. If $T'$ is a path, then we store $T'$ in memory. If $T'$ is not a path, then as in the proof of Theorem \ref{tree_con}, we note that there exists a vertex $u \in V(T')$ such that the graph obtained from $T'$ by removing $u$ has at least two connected components that are paths of order at least $1$. We choose such a vertex and label these paths $v_1 \dots v_j$ and $w_1 \dots w_k$ for $j, k \geq 1$ where $\dist(v_i, u) = i$ for each $1 \leq i \leq j$ and $\dist(w_i, u) = i$ for each $1 \leq i \leq k$. 

As in Theorem \ref{tree_con}, we can group all minimum path covers of $T'$ into three classes, where overlap is allowed between classes: Class 1 consists of minimum path covers containing $v_1 \dots v_j$, Class 2 consists of minimum path covers containing $w_1 \dots w_k$, and Class 3 consists of minimum path covers containing $v_j \dots v_1 u w_1 \dots w_k$.

Every minimum path cover of $T'$ is in at least one of the classes 1, 2, or 3, though some classes may be empty. Let $T_0$ be the graph obtained from $T'$ by removing $v_1 \dots v_j$, let $T_1$ be the graph obtained from $T'$ by removing $w_1 \dots w_k$, and let $T_2$ be the graph obtained from $T'$ by removing $v_j \dots v_1 u w_1 \dots w_k$. $T_0$ and $T_1$ are trees, so they are subtrees of $T$, and their order is less than $T'$, so we have already computed and stored in memory lists of minimum path covers for both $T_0$ and $T_1$. $T_2$ is a forest of order less than $T'$, so its connected components are subtrees of $T$ with order less than $T'$, so we can compute the minimum path covers of $T_2$ in $2^{O(n)}$ operations from the lists of minimum path covers stored in memory.

Using lists of minimum path covers for $T_0$, $T_1$, and $T_2$ from memory, we compute a list $L$ of minimum path covers for $T'$. We initialize $L$ as an empty list. Given a list of minimum path covers $S_0$ of $T_0$, we add to $L$ each of the path covers $\left\{s \cup \left\{ v_1 \dots v_j\right\}: s \in S_0 \right\}$. These path covers are not necessarily minimum path covers of $T'$, but they are the path covers in Class 1 if they are minimum, and Class 1 is empty if they are not minimum. Similarly given a list of minimum path covers $S_1$ of $T_1$, we add to $L$ the path covers $\left\{s \cup \left\{ w_1 \dots w_k\right\}: s \in S_1 \right\}$. These are the path covers in Class 2, if they are minimum path covers of $T'$. Similarly given a list of minimum path covers $S_2$ of $T_2$, we add to $L$ the path covers $\left\{s \cup \left\{ v_j \dots v_1 u w_1 \dots w_k\right\}: s \in S_2 \right\}$. These are the path covers in Class 3. Then we go through $L$, determine the minimum size of the path covers in $L$, remove any path covers that do not have minimum size, and remove any duplicates. Forming this new list takes $2^{O(n)}$ operations, and we store it in memory as the list of minimum path covers of $T'$.

At the end of the process, we compute a list of all minimum path covers of $T$ after computing a list of all minimum path covers of $T'$ for every proper subtree $T'$ of $T$. Since there are at most $2^n$ subtrees $T'$ of $T$ and computing a list of all minimum path covers of $T'$ takes $2^{O(n)}$ operations for each subtree $T'$ of $T$, computing a list of all minimum path covers of $T$ takes $2^{O(n)}$ operations.
\end{proof}

While we have computational results about the complexity of finding the set of all minimum zero forcing sets and the zero forcing graph, many fundamental computational questions about zero forcing graphs remain open. For example, is there an efficient algorithm to determine whether two minimum zero forcing sets are in the same connected component of the zero forcing graph? Similarly, is there an efficient algorithm to compute the distance between minimum zero forcing sets in the zero forcing graph?


\section{Acknowledgment}

The authors thank Jamie Haddock for asking the question that led to Theorem \ref{zfsleaf}.


\end{document}